\documentclass[12pt]{amsart}%
\usepackage{amsmath}
\usepackage[margin=1in]{geometry}
\usepackage{graphicx}
\usepackage{amsfonts}
\usepackage{amssymb}%
\usepackage[utf8]{inputenc}
\usepackage{comment}
\usepackage{hyperref}

\newtheorem{theorem}{Theorem}[section]
\theoremstyle{plain}

\newtheorem{claim}{Claim}

\newtheorem{lemma}{Lemma}[section]

\numberwithin{equation}{section}
\theoremstyle{definition}

\theoremstyle{remark}
\newtheorem{remark}{Remark}[section]

\begin{document}
\title[The Dirichlet problem]{
The Dirichlet problem for Lagrangian mean curvature equation}

\author{Arunima Bhattacharya}
\address{Department of Mathematics\\
University of Washington, Seattle, WA 98195, U.S.A.}
\email{arunimab@uw.edu}

\begin{abstract}
In this paper, we solve the Dirichlet problem with continuous boundary data for the Lagrangian mean curvature equation on a uniformly convex, bounded domain in $\mathbb{R}^n$.
\end{abstract}
\maketitle
\section{ Introduction}
	In this paper, we consider the Dirichlet problem for the Lagrangian mean curvature equation on a uniformly convex, bounded domain $\Omega\subset\mathbb{R}^n$, given by
	\begin{align}
		\begin{cases}
		F(D^{2}u)=\sum _{i=1}^{n}\arctan \lambda_{i}=\psi(x) \text{ in }  \Omega\\
		u=\phi \text{ on }
		\partial \Omega
		\end{cases}\label{lab}
		\end{align}
	where $\lambda_i$'s are the eigenvalues of the Hessian matrix $D^2u$, $\psi$ is the potential for the mean curvature of the Lagrangian submanifold $\{(x,Du(x))|x\in
\Omega\}\subseteq\mathbb {R}^n \times \mathbb {R}^n$, and $\phi$ is a given continuous function on $\partial \Omega$.

Our main results in this paper are the following:
	\begin{theorem}\label{main}
		\label{1}
		Suppose that $\phi\in C^{0}(\partial \Omega)$ and $\psi: \overline \Omega\rightarrow [(n-2)\frac{\pi}{2}+\delta, n\frac{\pi}{2})$ is in $C^{1,1}(\overline \Omega)$, where $\Omega$ is a uniformly convex, bounded domain in $\mathbb{R}^{n}$ and $\delta>0$. Then there exists a unique solution $u\in C^{2,\alpha}(\Omega)\cap  C^{0}(\partial \Omega)$ to the Dirichlet  problem (\ref{lab}).
	\end{theorem}
	\begin{theorem}\label{2.20}
		\label{1}
		Suppose that $\phi\in C^{0}(\partial \Omega)$ and $\psi:\overline \Omega \rightarrow (-n\frac{\pi}{2}, n\frac{\pi}{2})$ is a constant, where $\Omega$ is a uniformly convex, bounded domain in $\mathbb{R}^{n}$. Then there exists a unique solution $u\in C^{0}(\overline{\Omega})$ to the Dirichlet  problem (\ref{lab}).
	\end{theorem}

When the phase $\psi$ is constant, denoted by $c$, $u$ solves the special Lagrangian equation 
 \begin{equation}
\sum _{i=1}^{n}\arctan \lambda_{i}=c \label{s1}
\end{equation}
or equivalently,
\[ \cos c \sum_{1\leq 2k+1\leq n} (-1)^k\sigma_{2k+1}-\sin c \sum_{0\leq 2k\leq n} (-1)^k\sigma_{2k}=0.
\]
Equation (\ref{s1}) originates in the special Lagrangian geometry by Harvey-Lawson \cite{HL}. The Lagrangian graph $(x,Du(x)) \subset \mathbb {R}^n \times \mathbb {R}^n$ is called special
when the argument of the complex number $(1+i\lambda_1)...(1+i\lambda_n)$
or the phase $\psi$ is constant and it is special if and only if $(x,Du(x))$ is a (volume minimizing) minimal surface in $(\mathbb {R}^n \times \mathbb {R}^n,dx^2+dy^2)$ \cite{HL}.

A dual form of (\ref{s1}) is the Monge-Amp\'ere equation
\begin{equation*}
    \sum_{i=1}^n \ln\lambda_i=c.
\end{equation*}
This is the potential equation for special Lagrangian submanifolds in $(\mathbb {R}^n \times \mathbb {R}^n, dxdy)$ as interpreted in \cite{Hi}. The gradient graph $(x,Du(x))$ is volume maximizing in this pseudo-Euclidean space as shown by Warren \cite{W}. In the 1980s, Mealy \cite{Me} showed that an equivalent algebraic form of the above equation is the potential equation for his volume maximizing special Lagrangian submanifolds in $(\mathbb {R}^n \times \mathbb {R}^n, dx^2-dy^2)$.

A key prerequisite for the smooth solvability of the Dirichlet problem for fully nonlinear, elliptic equations is the concavity of the operator on the space of symmetric matrices.
The arctangent operator or the logarithmic operator is concave if $u$ is convex, or if the Hessian of $u$ has a lower bound $\lambda\geq 0$. Certain concavity properties of the arctangent operator are still preserved for saddle $u$. The concavity of the arctangent operator in (\ref{lab}) depends on the range of the Lagrangian phase. The phase $(n-2)\frac{\pi}{2}$ is called critical because the level set $\{ \lambda \in \mathbb{R}^n \vert \lambda$ satisfying $ (\ref{lab})\}$ is convex only when $|\psi|\geq (n-2)\frac{\pi}{2}$ \cite[Lemma 2.2]{YY}. The concavity of the level set is evident for $|\psi|\geq (n-1)\frac{\pi}{2}$ since that implies $\lambda>0$ and then $F$ is concave. For a supercritical phase $|\psi|\geq(n-2)\frac{\pi}{2}+\delta$ the operator $F$ can be extended to a concave operator \cite{CPX, CW}.

The Dirichlet problem for fully nonlinear, elliptic equations of the form $ F(\lambda[D^2u])=\psi(x)$ was studied by Caffarelli-Nirenberg-Spruck in \cite{CNS}, where they proved the existence of
classical solutions under various hypotheses on the function $F$ and the domain. Their results extended the work of Krylov \cite{kry1},  Ivo\v{c}kina \cite{Ivo1}, and their previous work \cite{CNS1} on equations of Monge-Ampère type. For the Monge-Ampère equation, continuous boundary data leads to only Lipschitz continuous solutions; Pogorelov \cite{P2} constructed his famous
counterexamples for the three dimensional Monge-Ampère equation $\sigma_3(D^2u) = \det(D^2u) = 1$, which also
serve as counterexamples for cubic and higher order symmetric $\sigma_k$ equations. In \cite{Tru}, Trudinger proved existence and a priori estimates of smooth solutions of fully nonlinear equations of the type of Hessian equations. In \cite{NT}, Ivo\v{c}kina-Trudinger-Wang studied the Dirichlet problem for a class of fully nonlinear, degenerate elliptic equations which depend only on the eigenvalues of the Hessian matrix. In \cite{HL0}, Harvey-Lawson studied the Dirichlet problem for fully nonlinear, degenerate elliptic equations of the form $F(D^2u)=0$ on a smoothly bounded domain in $\mathbb{R}^n$. Interior regularity for viscosity solutions of (\ref{s1}) with critical and supercritical constant phase $ |\psi|\geq(n-2)\frac{\pi}{2}$ was shown by Warren-Yuan \cite{WY}
and Wang-Yuan \cite{WaY}. For a subcritical phase $ |\psi|<(n-2)\frac{\pi}{2}$, singular solutions of (\ref{s1}) were constructed by Nadirashvili-Vl\u{a}du\c{t} \cite{NV} and Wang-Yuan \cite{WangY}. The existence and uniqueness of continuous viscosity solutions to the Dirichlet problem for (\ref{s1}) with continuous boundary data was shown by Yuan \cite{Ynotes}. In \cite{BrW}, Brendle-Warren studied a second
boundary value problem for the special Lagrangian equation. 
In \cite{CPX}, Collins-Picard-Wu solved the Dirichlet problem (\ref{lab}) on a compact domain with $C^4$ boundary value under the assumption of the existence of a subsolution and a supercritical phase restriction. In \cite{RB}, Dinew-Do-T{\^o} showed the existence and uniqueness of a $C^0$ solution to (\ref{lab}) on a bounded $C^2$ domain with $C^0$ boundary value under the assumption of the existence of a subsolution and a supercritical phase restriction.

In Theorem \ref{main}, we assume $\psi\geq (n-2)\frac{\pi}{2}+\delta$ since
by symmetry $\psi\leq- (n-2)\frac{\pi}{2}-\delta$ can be treated similarly. The proof of Theorem \ref{main} follows from a standard continuity method and a uniform approximation of the $C^0$ boundary value. The major difficulty in proving uniform $C^{2,\alpha}$ estimates up to the boundary, which is necessary for the continuity method, is in estimating the double normal derivatives at the boundary without the aid of a given subsolution. We get around this by constructing a lower linear barrier function for $u_n$ by applying Trudinger's technique and a change of basis argument. Once we derive uniform $C^{2,\alpha}$ estimates up to the boundary, we use the a priori interior Hessian estimates proved in \cite{AB} to approximate the $C^0$ boundary value from which Theorem \ref{main} follows. In Theorem \ref{2.20}, we consider all values of the constant Lagrangian phase, which includes subcritical values. The main difficulty here is the lack of uniform ellipticity and concavity. The proof follows via Perron's method using an idea that was introduced by Ishii \cite{Ish2}
where we apply comparison principles for strictly elliptic\footnote{$F(D^2u)=\psi$ is strictly elliptic in the sense that $(F_{u_{ij}}(D^2u))>0$}, non concave, fully nonlinear equations \cite{yy2004}. In \cite{HL0}, Harvey-Lawson established the existence and uniqueness of continuous solutions of fully nonlinear, degenerate elliptic equations of the form $F(D^2u)=0$ on a smoothly bounded domain in $\mathbb{R}^n$ under an explicit geometric $F$‐convexity assumption on the boundary of the domain. The key ingredients of their proof were the usage of subaffine functions and Dirichlet duality. As an application, the continuous solvability of the constant phase equation (\ref{s1}) is obtained. In contrast, in Theorem \ref{2.20} of this paper, we focus only on the continuous solvability of the Dirichlet problem of equation (\ref{s1}) and provide a short proof that solely relies on a certain comparison principle.

\begin{remark}For Theorem \ref{main}, an assumption weaker than $C^{1}$ on $\psi$ will lead to counterexamples with continuous boundary data.
For example, in two dimensions, we consider a boundary value problem of (\ref{lab}) on the unit ball $B_1(0)$ where the phase is in $C^{\alpha}$ with $\alpha\in (0,1)$: $\psi(x)=\frac{\pi}{2}-\arctan (\alpha^{-1}|x|^{1-\alpha})$ and $u(x)=\int_{0}^{|x|}t^{\alpha}dt$ on $\partial B_1$. This problem admits a non $C^2$ viscosity solution $u$ with gradient $D u=|x|^{\alpha-1}x$, thereby proving a contradiction. If the Lagrangian phase is subcritical, i.e. $ |\psi|<(n-2)\frac{\pi}{2}$, then even for the constant phase equation (\ref{s1}) with analytic boundary data, $C^0$ viscosity solutions may only be $C^{1,\varepsilon_{0}}$ but no more as shown by Wang-Yuan \cite{WangY}.\\
However, the existence of $C^{2,\alpha}$ solutions to (\ref{lab}) with critical and supercritical phase, i.e. $|\psi|\geq (n-2)\frac{\pi}{2}$, where $\psi\in C^{1,\varepsilon_{0}}$ or even or even $|\psi|\geq (n-2)\frac{\pi}{2}$ where $\psi\in C^{1,1}$, are still open questions. As of now, it is also unknown if $C^0$ viscosity solutions of (\ref{s1}) are Lipschitz for subcritical phases.
\end{remark}

\begin{remark}
 In Theorem \ref{2.20}, if we replace the constant phase with any continuous function lying in the subcritical or critical range, then the existence and uniqueness of $C^{0}$ viscosity solutions of (\ref{lab}) remain open questions. This is due to the lack of a suitable comparison principle for strictly elliptic, non concave, fully nonlinear equations with a variable right hand side. In \cite{hl1}, Harvey-Lawson introduced a condition called “tameness” on the operator $F$, which is a little stronger than strict ellipticity and allows one to prove comparison. In \cite{hl2}, they further proved that for the Lagrangian mean curvature equation, one can only show tamability in the supercritical phase interval. Recently in \cite{kev}, Cirant-Payne established comparison for this equation when the range of the phase is restricted to the intervals $((n-2k)\frac{\pi}{2},(n-2(k-1))\frac{\pi}{2})$ where $1\leq k\leq n$. This in turn solves the Dirichlet problem on these intervals as shown in \cite[Theorem 6.2,C]{hl2}.
 \\
 For $\sigma_k$ equations with a variable right hand side, results analogous to Theorem \ref{2.20} exist. This is due to the fact that the linearized operator has a positive lower bound in determinant unlike the Lagrangian mean curvature equation (\ref{lab}).
 \end{remark}

This article is divided into the following sections: in section two, we state some well known algebraic and trigonometric inequalities satisfied by solutions of (\ref{lab}). In section three, we prove $C^{2,\alpha}$ estimates up to the boundary assuming $C^4$ boundary data. In section four, we first solve the Dirichlet problem with $C^{4}$ boundary data using the method of continuity and then combine it with the Hessian estimates proved in \cite{AB} to solve the Dirichlet problem with continuous boundary data. In section five, we prove Theorem \ref{2.20}. In section six (appendix), we state a well known linear algebra Lemma that we use in estimating the Hessian of $u$ on the boundary and we provide the proof of a certain comparison principle that is essential for the proof of Theorem \ref{2.20}.\\

\textbf{Acknowledgments.} The author is grateful to Y.Yuan for his guidance, support, and several useful discussions. The author is grateful to R.Harvey and B.Lawson for their insightful feedback on the comparison principle. The author thanks R.Shankar and M.Warren for helpful comments and suggestions. 

\section{Preliminaries}

The induced Riemannian metric on the Lagrangian submanifold $\{(x,Du(x))|x\in \Omega\}\subset \mathbb {R}^n \times \mathbb {R}^n$ is given by
\[g=I_n+(D^2u)^2.
\]
On taking the gradient of both sides of the Lagrangian mean curvature equation (\ref{lab}) we get
\begin{equation}
\sum_{a,b=1}^{n}g^{ab}u_{jab}=\psi_j \label{linearize}
\end{equation}
where $g^{ab}$ is the inverse of the induced Riemannian metric $g$. From \cite[(2.19)]{HL} we see that the mean curvature vector $\vec{H}$ of this Lagrangian submanifold $\{(x,Du(x))|x\in\Omega\}$
is given by $\vec{H}=J\nabla_g\psi $ where $\nabla_g$ is the gradient operator for the metric $g$ and $J$ is the complex structure, or the $\frac{\pi}{2}$ rotation matrix in  $\mathbb{R}^n\times \mathbb{R}^n$. Next we state the following Lemma.
\begin{lemma}\label{y1}
		Suppose that the ordered real numbers $\lambda_{1}\geq \lambda_{2}\geq...\geq \lambda_{n}$ satisfy (\ref{lab}) with $\psi\geq (n-2)\frac{\pi}{2}$.
		Then we have \begin{enumerate}
			\item $\lambda_{1}\geq \lambda_{2}\geq...\geq \lambda_{n-1}>0, \lambda_{n-1}\geq |\lambda_{n}|$,
			\item $\lambda_{1}+(n-1)\lambda_{n}\geq 0$,
			\item $\sigma_{k}(\lambda_{1},...,\lambda_{n})\geq 0$ for all $1\leq k\leq n-1$ and $n\geq 2$,
			\item  if $\psi\geq (n-2)\frac{\pi}{2}+\delta$, then $D^2u\geq -\cot \delta I_n$.
			
		\end{enumerate}
		
	\end{lemma}
\begin{proof}
Properties (1), (2), and (3) follow from \cite[Lemma 2.1]{WaY}. Property (4) follows from \cite[Pg 1356]{YY}.
\end{proof}

	\section{$C^{2,\alpha}$ estimate up to the boundary}
	We first prove the following $C^{2,\alpha}$ estimate up to the boundary of $\Omega$. 
	\begin{theorem}
		\label{2}
		Let $\phi\in C^{4}(\overline \Omega)$ and $\psi:\overline \Omega\rightarrow [(n-2)\frac{\pi}{2}+\delta, n\frac{\pi}{2})$ be in $C^{2,\alpha}(\overline \Omega)$, where $\Omega$ is a uniformly convex domain in $\mathbb{R}^{n}$ and $\delta>0$.
		Then there exists a universal constant $\alpha\in(0,1)$ such that if $u\in C^{4,\alpha}(\overline \Omega)$ is a solution of (\ref{lab}), then
		\begin{equation}||u||_{C^{2,\alpha}(\overline \Omega)}\leq C(||\psi||_{C^{1,1}(\overline \Omega)}, ||\phi||_{C^{4}(\overline \Omega)},n, \delta,\overline{\Omega}). \label{bdy}
		\end{equation}
	\end{theorem}
	
	\begin{proof}
	We first make the following observation, which will be used for steps 1,2,3.2, and 3.3 below.\\
	We pick an arbitrary boundary point $x_{0}\in\partial\Omega$. By a rotation and translation we choose a co-ordinate system such that the chosen boundary point is the origin and $\Omega$ lies above the hyperplane $\{x_{n}=0\}$ with $e_n$ as the inner unit normal at $0$. For such a domain, we can write \begin{equation}
	    \partial\Omega=\{(x',x_n)|x_n=h(x')=\frac{1}{2}(k_1x_1^2+...+k_{n-1}x_{n-1}^2)+o(|x'|^2)\}. \label{lala}
	\end{equation}
	
	 At $0\in\partial \Omega$ the boundary value satisfies \begin{align}
	\phi(x',x_n)=\phi(x',h(x'))=\phi(0)+\phi_{x'}(0)x'\nonumber\\
	+\phi_{x_n}(0)h(x')+\phi_{x'x'}(0)x'x'+\phi_{x_nx_n}(0)h(x')h(x')+o(|x'|^2+h^2(x'))\nonumber\\
	=Q(x)+o(1)|x'|^2 \nonumber
	\end{align}
	where $Q(x)$ is a quadratic. So there exists $C_0=C_0(||\phi||_{C^2(\partial{\Omega})}, n,\kappa)$ such that 
	\begin{align}
	L^-=-C_0x_n\leq \phi\leq C_0x_n=L^+ \text{ on $\partial\Omega$}. \label{lala1}
	\end{align}\\
	We now prove estimate (\ref{bdy}) in the following four steps. We will estimate all the boundary derivatives of $u$ at the origin. 
		\begin{itemize}
		\item[Step 1.] Bound for $||u||_{L^{\infty}(\overline{\Omega})}$.
		\begin{claim} We show the following
		\begin{equation}
		||u||_{L^{\infty}(\overline{\Omega})}\leq C(||\phi||_{C^{2}(\overline \Omega)},n, |\partial \Omega|_{C^2}) .\label{est1}
		\end{equation} 
		\end{claim}
		\begin{proof}The function $\psi:\overline \Omega\rightarrow [(n-2)\frac{\pi}{2}+\delta, n\frac{\pi}{2})$ is in $C^{1,1}(\overline \Omega)$, so there exists $\varepsilon>0$ such that $\psi<n\frac{\pi}{2}-\varepsilon$. Fixing this $\varepsilon$ we define  $\underline{\psi}=(n-2)\frac{\pi}{2}+\delta$ and $\overline{\psi}=n\frac{\pi}{2}-\varepsilon$.\\
		Recalling (\ref{lala1}) we find constants $c_0$ and $C'_0$ depending on $C_0$ above such that on $\partial \Omega$, we have 
		\begin{equation*}
		-c_0|x|^2+ \frac{1}{2}|x|^2\tan\frac{\overline{\psi}}{n}=-C'_0|x|^2=-C_0x_n\leq \phi \leq C'_0|x|^2+ \frac{1}{2}|x|^2\tan\frac{\underline{\psi}}{n}.
		\end{equation*}
		Using relation (\ref{lala}) we define
		\begin{align}
		-Cx_n+ \frac{1}{2}|x|^2\tan\frac{\overline{\psi}}{n}
		=B^-\label{bn}\\ 
		Cx_n+ \frac{1}{2}|x|^2\tan\frac{\underline{\psi}}{n}
		=B^+.\label{bbnn}
		\end{align} where $C=C(||\phi||_{C^2(\partial{\Omega})}, n,\kappa_i)$.
		We observe that
		\begin{align}
		F(D^{2}B^-)\geq F(D^{2}u) \geq F(D^{2}B^+) \text{ in $\Omega$}\nonumber\\
		B^-\leq u\leq B^+ \text{ on $\partial\Omega$ with equality holding at $0$.} \label{lebu}
\end{align}
		Using comparison principles we see that (\ref{est1}) holds.
	\end{proof}
	\item[Step 2.] Bound for $||Du||_{L^{\infty}(\overline{\Omega})}$.
		\begin{claim} We show the following
			\begin{equation}
			||Du||_{L^{\infty}(\overline{\Omega})}\leq C(||\psi||_{C^{1}(\overline \Omega)},||\phi||_{C^{2}(\overline \Omega)},n, \delta,|\partial \Omega|_{C^2}). \label{1.1}
			\end{equation}
		\end{claim}
	\begin{proof}
		
		On linearizing (\ref{lab}), we get (\ref{linearize}) and since $\psi\in C^{1,1}(\overline{\Omega})$, we see that $|g^{ij}\partial_{ij}u_e|\leq C(|\psi|_{C^1(\overline{\Omega})})$. From Lemma (\ref{y1}), we see that u is semi-convex, i.e. $D^2u\geq -\cot\delta I_n$. We modify $u$ to the convex function $u+\cot\delta \frac{|x|^2}{2}$ from which we see that $Du(x)+\cot\delta x$ attains its supremum on the boundary of $\Omega$. So we have
		\begin{equation}
		\sup_{\bar{\Omega}}Du(x)=\sup_{\partial \Omega}Du(x)+\cot\delta. \label{1h}
		\end{equation}
		For $1\leq i\leq n-1$, $u_{i}=\phi_{i}$, so we only need to estimate $u_{n}(0)$. 
		Recalling (\ref{lebu}), we 
		again use comparison principles and on taking the normal derivative at $0$, we get
		\[|u_{n}(0)|\leq C(||\psi||_{C^{1}(\overline \Omega)},||\phi||_{C^{2}(\overline \Omega)},n,|\partial \Omega|_{C^2}).
		\]
		Combining (\ref{1h}) with the above we get (\ref{1.1}).
	\end{proof}
	\item[Step 3.] Bound for $||D^{2}u||_{L^{\infty}(\overline \Omega)}$. 
		\begin{claim} We prove the following 
		\begin{equation}||D^{2}u||_{L^{\infty}(\overline \Omega)}\leq C(||\psi||_{C^{1,1}(\overline \Omega)}, ||\phi||_{C^{4}(\overline \Omega)},n,\delta,|\partial \Omega|_{C^4}) .\label{3}
		\end{equation}
	\end{claim}

	The proof of the above claim follows from the following steps.
	
		\item[Step 3.1] We first prove that the Hessian attains its supremum on the boundary of $\Omega$. We show that
		\begin{equation}
		||D^{2}u||_{L^{\infty}(\overline \Omega)}\leq C(||\psi||_{C^{1,1}(\overline \Omega)}, ||D^{2}u||_{L^{\infty}(\partial \Omega)},\delta). \label{4}
		\end{equation}
		We differentiate (\ref{lab}) twice and since the phase is supercritical we modify the operator $F$ to a concave operator $\Tilde{F}$ as shown in \cite[pg 347]{CW}. We see that
		\begin{align*}\Tilde{F}^{ij}\partial_{ij}u_{ee}+\Tilde{F}^{ij,kl}\partial_{ij}u_{e}\partial_{kl}u_{e}=\psi_{ee}\\
		\Tilde{F}^{ij}\partial_{ij}\Delta u=\Delta \psi -\sum_{e}\Tilde{F}^{ij,kl}\partial_{ij}u_{e}\partial_{kl}u_{e}\geq \Delta \psi .
		\end{align*}
		The last inequality follows from the concavity of the operator. Let $p_0$ be an interior point of $\Omega$. By an orthogonal transformation, we assume $D^2u$ to be diagonalized at $p_0$. We observe that \begin{align*}
		g^{ij}\partial_{ij}(\Delta u+\frac{C_{1}}{2}|x|^{2})(p_0)\geq \\
		-C(||\psi||_{C^{1,1}(\Omega)})+C_1\sum_{i=1}^n\frac{1}{1+\lambda_i^2}>0
	\end{align*}
	where $C_1$ is chosen large enough using the semi-convexity of $u$. 
		This shows that $D^2u$ attains its supremum on the boundary. 
		Next, we estimate the Hessian on the boundary in the following steps.

			\item[Step 3.2] We estimate the double tangent derivative $u_{TT}$ on the boundary.\\
		That is we estimate $u_{ik}(0)$ on $\partial \Omega$ for $1\leq i,k\leq n-1$. 
        There exists a constant $a>0$ for which $(x_1,...,x_{n}-a)$ is orthogonal to $\partial \Omega$ near $0$. Consider the following tangential derivative near $0\in \partial \Omega$
         \[\partial_{T_{k} }u(x)=(a-x_n)u_k(x)+x_ku_n(x).
        \]
        For $1\leq k\leq n-1$ we have 
        $\partial_ {T_{k}}u|_{\partial \Omega}=\partial _{T_{k}}\phi|_{\partial \Omega}$.
        So for $1\leq k,i\leq n-1$, we have
        \begin{align*}
            \partial_ {T_{k}}u(0)=au_k(0)\\
           \partial_ {T_{i}} \partial _{T_{k}}u(0)=au_{ki}(0)+\delta_{ki} u_n(0)\\
           \implies \partial_{T_{i}} \partial_ {T_{k}}\phi(0)=au_{ki}(0)+\delta_{ki} u_n(0).
        \end{align*}
		Using the estimate in step 2, we have \begin{equation*}|u_{ki}(0)|\leq C(||\psi||_{C^{1}(\overline \Omega)}, ||\phi||_{C^{2}(\overline \Omega)},n,\delta,\Omega). 
		\end{equation*}

		\item[Step 3.3] We estimate the mixed tangent normal derivative $u_{TN}$ on the boundary.\\
		That is we estimate $u_{in}(0)$ on $\partial \Omega$ for $1\leq i\leq n-1$.
		Let $\tau$ be a vector field generated by rotation such that $\tau(0)=e_i$ for $i<n.$  Note that $u_\tau=\phi_\tau$ on $\partial \Omega$ and $g^{ij}\partial_{ij}u_\tau=\psi_\tau$ in $\Omega$.\\
		Applying the argument in (\ref{lala1}) we get the following on $\partial \Omega$ 
		\begin{align}
		    |\phi_{\tau}|\leq C(||\phi_\tau||_{C^{2}(\overline{\Omega})},||\psi||_{C^{1}(\overline \Omega)},n,|\partial \Omega|_{C^2})x_n 
		    \leq C|x|^{2}\label{nn}.
		\end{align}
		Using the above we choose a constant $c>0$ depending on $C$ above such that 
	\begin{equation*}
		-c|x|^2+ \frac{1}{2}|x|^2\tan\frac{\overline{\psi}}{n}=-C|x|^2\leq \phi \leq Cx_n+ \frac{1}{2}|x|^2\tan\frac{\underline{\psi}}{n} \text{ on $\partial\Omega$}.
		\end{equation*}
		We define $u_0$ to be the subsolution
		\[u_0=-C'x_n+ \frac{1}{2}|x|^2\tan\frac{\overline{\psi}}{n}\]
		where $C'=C'(||\phi||_{C^{3}(\overline{\Omega})},||\psi||_{C^{1}(\overline \Omega)},n,|\partial \Omega|_{C^2})$. Let $w=u-u_0$. Since the phase lies in the supercritical range, we again extend the operator $F$ to a concave operator. Using concavity we get the following for some $\varepsilon_0>0$
		on a small ball of radius $r$ around the origin  \begin{align}
		g^{ij}w_{ij}\leq-\varepsilon_{0}  \text{ inside } \Omega \cap B_{r}(0)\nonumber \\
		w\geq 0 \text{ on } \partial(\Omega\cap B_{r}(0))\nonumber \\
		w(0)=0. \label{3.7} 
		\end{align} 
		We now choose $\alpha,\beta$ large such that 
		\begin{align}
		    g^{ij}\partial_{ij}(\alpha w+\beta |x|^2 \pm u_\tau)\leq 0 \text{ in } \Omega \cap B_{r}(0) \label{i}\\
		    \alpha w+\beta|x|^2 \pm u_\tau\geq 0 \text{ on } \partial(\Omega\cap B_{r}(0)).\nonumber
		\end{align} 
		Since $w\geq 0$ on $\partial(\Omega\cap B_{r}(0))$ we only need to choose $\beta$ large such that \[\beta|x|^2\pm u_\tau\geq0 \text{ on } \partial(\Omega\cap B_{r}(0)).\] 
		
		    We observe that on $\Omega\cap \partial B_{r}(0)$,
		    $\beta \geq \frac{C}{r^{2}}$
		 where $C= C(||\psi||_{C^{1}(\overline \Omega)}, ||\phi||_{C^{2}(\overline \Omega)},\delta,n,|\partial \Omega|_{C^2}) $ is from the gradient estimate in (\ref{1.1}). And on applying (\ref{nn}) we get the required value of $\beta$ on  $\partial \Omega\cap  B_{r}(0)$.
		Fixing the larger of the two values to be the constant $\beta$ we now choose $\alpha$ such that (\ref{i}) holds good. We have 
		\begin{align*}
		    g^{ij}\partial_{ij}(\alpha w+\beta|x|^2 \pm u_\tau)\leq -\alpha\varepsilon_{0} +C 
		    \end{align*} 
		    where $C=C(\beta,|\psi|_{C^{1}(\overline{\Omega})})$. We choose $\alpha$ large such that
		 $-\alpha\varepsilon_{0} +C\leq 0$.
	Observe that $\alpha w+\beta|x|^2\pm u_\tau(0)=0$ at $0$. Using Hopf's Lemma we see that 
	\begin{align*}
	    \partial_n(\alpha w+\beta|x|^2\pm u_\tau)(0)\geq 0\\
	   \implies \pm u_{\tau_{n}(0)}\geq \mp \partial_n(\alpha w+\beta|x|^2\pm u_\tau)(0)\\
	    \implies |u_{\tau_{n}}(0)|\leq |\alpha w_n(0)|\leq C.
	\end{align*}
		Therefore, we have \begin{equation*}|u_{in}(0)|\leq C(||\psi||_{C^{1,1}(\overline \Omega)}, ||\phi||_{C^{3}(\overline \Omega)},n,\delta,|\partial \Omega|_{C^2}). 
		\end{equation*}
		\item[Step 3.4] Lastly, we estimate the double normal $u_{NN}$ on the boundary.\\
		Note that by Lemma \ref{y1}, $D^2u$ is bounded below, so we only need to prove an upper bound for $u_{NN}$, which we find using an idea of Trudinger \cite{Tru}. \\Suppose that
$\lambda'$ denotes the eigenvalues of the $n-1\times n-1$ matrix $u_{TT}$ where the tangent vector $T$ acts as
	\[u_{TT}=\frac{1}{r^2}u_{\theta\theta}+\frac{1}{r}u_{r},
	\] when the boundary is a sphere. We denote
	\[D^2u = \begin{bmatrix} u_{TT} & u_{T\gamma} \\u_{\gamma T} & u_{\gamma \gamma}\end{bmatrix}=\begin{bmatrix} \lambda' & u_{T\gamma} \\u_{\gamma T} & u_{\gamma\gamma}\end{bmatrix}.\]
Let $x'_{0}$ be the minimal point of $\Tilde{\Theta}(\lambda')\vert_{\partial \Omega}$ where 
	\[\Tilde{\Theta}(\lambda')=\sum_{i=1}^{n-1}\arctan \lambda_i'-\psi
	\] and we denote $\lambda'_{0}=\lambda'(x'_0).$ Our goal is to find a lower linear barrier function for $u_\gamma$ at $x'_0$ followed by the same for $u_n$ at $x'_0$ with the help of a change of basis technique. Using this we will find an upper bound of $u_{nn}(x_0')$ followed by an upper bound of $u_{nn}(x)$ for all $x\in\partial\Omega$.\\
	Now we estimate the lower bound of $tr(D^2u)|_{T}=\sum_{i=1}^{n-1}\lambda'_{i}$. Observe that $\Tilde{\Theta}(\lambda')\geq\Tilde{\Theta}(\lambda'_0)>\psi-\frac{\pi}{2}>(n-3)\frac{\pi}{2}$. So the level set $\{\lambda'\in \mathbb{R}^{n-1}| \Tilde{\Theta}(\lambda')=\Tilde{\Theta}(\lambda'_0)\}$ should be convex. Heuristically, this property means the following: 
	\begin{align*}
	    \langle D\Tilde{\Theta}(\lambda'_0),\lambda' \rangle \geq \langle D\Tilde{\Theta}(\lambda'_0),\lambda'_0 \rangle=K_0 \text{ and } "="\text{ at } x'_0
	    \end{align*} where $K_0$ is a constant depending on $|\psi|_{C^{1}(\Omega)}, |\phi|_{C^2(\partial \Omega)}$, and $\delta$. Denoting \[\Bigg[\frac{\partial\Tilde{\Theta}(D^2u(x_0))|_{T}}{\partial D^2u|_T} \Bigg]=A_{ij}(\lambda'_0)\] where $1\leq i,j\leq n-1$,  we see that 
	    \begin{align*}
	    tr (A_{ij}(\lambda'_0)) (D^2u(x)|_T)\geq K_0 \text{ with equality holding at } x'_0.
	\end{align*}
	 Denoting the second fundamental form by $II$, we observe that
	 	\begin{align}
	D^2(u-\phi)|_T=(u-\phi)_{\gamma} II|_{\partial \Omega}\nonumber\\
	tr[A_{ij}(\lambda'_0)(D^2\phi|_T-\phi_\gamma II|_{\partial \Omega}+u_\gamma II|_{\partial \Omega})] \geq K_0 \text{ with equality holding at } x'_0. \nonumber
	  \end{align}
	This shows
	\begin{align}
	u_\gamma\geq \frac{1}{\sum_{i=1}^{n-1}\Tilde{\Theta}_i(\lambda'_0)\kappa_i(x')}[K_0-tr(A_{ij}(\lambda'_0)(D^2\phi|_T-\phi_{\gamma}II|_{\partial\Omega}))]\text{ with equality holding at } x'_0\label{above}\\
	\implies u_\gamma\geq C(|\phi|_{C^4(\overline{\Omega})}, |\partial \Omega|_{C^4}, |\psi|_{C^{1}(\Omega)},\delta) \text{ with equality holding at } x'_0\nonumber
	\end{align}
	where the last inequality follows from the observation that for all the terms in the LHS of (\ref{above}) one can find a lower linear barrier function whose Lipschitz norm depends on the $C^{3,1}$ norm of $\phi$ and the $C^1$ norm of $\psi$.
	Next, we consider a unit local basis at $x_0'$ denoted by $\mathcal{B}=\{e_n, e_{T_{\alpha}}| 1 \leq \alpha\leq n-1 \}$ where $e_n$ is the outward unit normal and $e_{T_\alpha}$ represents vectors in the tangential direction at $x_0'$. By a change of basis we write the unit radial direction vector $e_\gamma$ as 
	$e_\gamma=ae_n+be_{T_{\alpha}}.$ A simple computation shows that\begin{align*}
	    e_\gamma=\frac{\langle e_\gamma,e_n\rangle}{1-\langle e_n, e_{T_{\alpha}}\rangle ^2}e_n-\frac{\langle e_\gamma,e_n\rangle \langle e_n, e_{T_{\alpha}}\rangle}{1-\langle e_n, e_{T_{\alpha}}\rangle ^2}e_{T_{\alpha}}
	\end{align*}
	from which one can easily find a lower linear barrier for $u_n$ at $x_0'$. 
	 So far we have
		\begin{align}
		u_n\geq L_1^{-}(x',x_n)\text{ on }\partial \Omega \text{ with equality holding at $x'_0$} \label{ma}
		\end{align}
		 where 
		 \begin{equation*}
		     L_1^{-}(x',x_n)=-C(|\phi|_{C^4}, |\partial \Omega|_{C^4}, |\psi|_{C^{1}(\Omega)},\delta)x_n\geq -C|x|^2 .
		 \end{equation*} 
		 Now we choose coordinates such that $x'_{0}$ is the
origin and the $n-1\times n-1$ matrix $u_{TT}(0)$ is diagonalized. 
		\begin{claim}\label{3.0} We show that
		\[u_{nn}(0)\leq C\] where $C=C(||\psi||_{C^{1,1}(\overline \Omega)},  ||\phi||_{C^{4}(\overline \Omega)},n,\delta, |\partial \Omega|_{C^4} )$.\\ Note that unlike before $e_n$ is the outward unit normal now. 
	\end{claim}
	\begin{proof}
Note that $u_n$ is a solution of the linearized equation $g^{ij}D_{ij}u_{n}=\psi_n$, so we get
\begin{equation}|g^{ij}\partial_{ij}u_n|\leq C(||\psi||_{C^{1}(\Omega)}). \label{h}
		\end{equation}  
		Now we repeat the process in step 3.3. We define $w=u-B^{-}$ where $B^{-}$ is the subsolution defined in (\ref{bn}) and we see that $w$ satisfies condition (\ref{3.7}). We choose $\alpha$ and $\beta$ large such that 
		\begin{align}
		g^{ij}\partial_{ij}(\alpha w+\beta|x|^2 + u_n)\leq 0 \text{ in }\Omega \cap B_{r}(0) \label{i1}\\
		 \alpha w+\beta|x|^2+ u_n\geq 0 \text{ on }\partial(\Omega\cap B_{r}(0)).\nonumber
		\end{align}
		As $w\geq 0$ on $\partial (B_r(0)\cap \Omega))$ we first choose $\beta$. On $\partial B_{r}(0)\cap \Omega$, we have $\beta\geq -C/r^{2} $
	where $C= C(||\psi||_{C^{1}(\overline \Omega)}, \delta, ||\phi||_{C^{2}(\overline \Omega)},n,|\partial \Omega|_{C^2}) $ is the constant from the estimates in (\ref{1.1}) and (\ref{est1}).
On $\partial \Omega\cap  B_{r}(0)$, we find $\beta$ using (\ref{ma}). Choosing the larger of the two values we get the required value of $\beta$. 
	 Fixing this $\beta$ we choose $\alpha$ such that (\ref{i1}) holds. Using the constant $C$ from (\ref{h}), we choose $\alpha$ large such that $-\alpha\varepsilon_{0}+C<0$
		where $C=C(\beta,||\psi||_{C^{1}(\overline{\Omega})})$.
		Now since $(\alpha w+\beta|x|^2+u_n)(0)=0$, using Hopf's Lemma we get
		\begin{align*}
		\frac{\partial}{\partial_n}(\alpha w+\beta|x|^2+u_n)(0)\leq 0\\
		\implies u_{nn}(0)\leq C(||\psi||_{C^{1,1}(\overline \Omega)},  ||\phi||_{C^{4}(\overline \Omega)},n,\delta, |\partial \Omega|_{C^4} ).
		\end{align*}
		
\end{proof}
	Next we prove the following claim:
	\begin{claim} If $u_{nn}(0)$ is bounded above, then $u_{nn}(x)$ will be bounded above for all $x\in \partial \Omega$.
 	\end{claim}
\begin{proof}  
Suppose that for some $x_p\in \partial \Omega$, $u_{nn}(x_p)\geq K$ where $K$ is a large constant to be chosen shortly. 
From claim \ref{3.0}, we see that
at $0$, 
\begin{align*}
F(D^2u+N e_n \times e_n)-F(D^2u)=\delta_0 (||\phi||_{C^4(\partial \Omega)},||\psi||_{C^{1,1}(\overline \Omega)})>0\\
\implies \lim_{a\rightarrow\infty} F(D^2u+a e_n \times e_n)\geq F(D^2u+N e_n \times e_n)
\geq F(D^2u)+\delta_0
=\psi +\delta_0.
\end{align*}
From Lemma \ref{3.9}, we see that 
\[ \sum_{i=1}^{n-1} \arctan \lambda'_i(x_{p})\geq \psi+\delta_0-\frac{\pi}{2}
\]
and
\begin{align*}
\psi= F(D^2u)=\sum_{i=1}^{n-1}\arctan\lambda'_i+o(1)+\arctan(u_{nn}+O(1))\\
\geq \psi +\delta_0-\frac{\pi}{2}-\frac{\delta_0}{2}+\arctan(u_{nn}+O(1)).
\end{align*}
Now if we choose $K$ large enough such that 
\[ u_{nn}(x_p)>\tan(\frac{\pi}{2}-\frac{\delta_0}{2})-O(1)
\] we arrive at a contradiction. 
Therefore, choosing\\ $K\leq\tan(\frac{\pi}{2}-\frac{\delta_0}{2})-O(1)=C(||\psi||_{C^{1,1}(\overline \Omega)},  ||\phi||_{C^{4}(\overline \Omega)},n,\delta, |\partial \Omega|_{C^4} )$, we see that $u_{nn}(x)\leq K$ for all $x\in \partial \Omega$. Combining all the estimates in step 3 above we obtain (\ref{3}).
\end{proof}
\item[Step 4.] Bound for $||D^2u||_{C^{\alpha}(\overline{\Omega)}}$.\\
This follows from the interior $C^{2,\alpha}$ estimates by Evans-Krylov \cite{EK,Kr} and the boundary $C^{2,\alpha}$ estimates by Krylov \cite[Theorem 4.1]{Kr}. \\
Therefore, combining all the four steps above we obtain estimate (\ref{bdy}).

\end{itemize}
\end{proof}
\section{Proof of Theorem \ref{main}}
	In this section we use the $C^{2,\alpha}$ estimate up to the boundary to solve the following Dirichlet problem using the method of continuity.
	
	\begin{theorem}
		\label{1}
		Suppose that $\phi\in C^{4}(\overline \Omega)$ and $\psi: \overline \Omega\rightarrow [(n-2)\frac{\pi}{2}+\delta, n\frac{\pi}{2})$ is in $C^{1,1}(\overline \Omega)$ where $\Omega$ is a uniformly convex, bounded domain in $\mathbb{R}^{n}$ and $\delta>0$. Then there exists a unique solution $u\in C^{2,\alpha}(\overline{\Omega})$ to the Dirichlet  problem (\ref{lab}).
	\end{theorem}
	\begin{proof}
	For each $t\in[0,1]$, consider the family of equations
	\begin{align}
	\begin{cases}
		F(D^{2}u)=t\psi +(1-t)c_0\text{ in }  \Omega\\
		u=\phi \text{ on }\partial \Omega \label{eq2}
		\end{cases}
		\end{align} where $c_0=(n-2)\frac{\pi}{2}+\delta$ and $\psi\in C^{2,\alpha}(\overline{\Omega})$.
	Let $I=\{t\in[0,1]\vert$ there exists $u_t\in C^{4,\alpha}(\overline{\Omega})$ solving (\ref{eq2})$\}$. We know that $0\in I$ from \cite{YYY}. The fact that $I$ is open is a consequence of the implicit function Theorem and invertibility of the linearized operator (\ref{linearize}). The closedness of $I$ follows from the apriori estimates. Hence, $1\in I$. Now using a smooth approximation\footnote{When $\psi$ is in $C^{1,1}(\overline{\Omega})$ 
we can take a sequence of smooth functions $\psi_k$ approximating $\psi$ and 
a sequence of solutions $u_k$ solving (\ref{lab}) with $\psi_k$ as the right hand side. Applying the uniform $C^{2,\alpha}$ estimate and taking a limit
solves the equation.} we solve (\ref{lab}) for $\psi\in C^{1,1}$. Uniqueness follows from the maximum principle for fully nonlinear equations.
	\end{proof}
	\begin{remark}There exists a unique smooth solution to the Dirichlet
problem (\ref{lab}) if all data is smooth and if the phase lies in the supercritical range.
\end{remark}
\begin{proof} of \textbf{Theorem \ref{main}}\\
	We approximate $\phi\in C^0(\partial\Omega)$ uniformly on $\partial \Omega$ by a sequence $\{\phi_k\}_{k\geq 1}$ of $C^4$ functions and solve 
	\begin{align*}
		\begin{cases}
		F(D^{2}u_k)=\psi \text{ in }  \Omega\\
		u_k=\phi_k \text{ on } \partial \Omega
		\end{cases}
		\end{align*} using Theorem 4.1.
	 Applying the interior Hessian estimates proved in \cite[Theorem 1.1]{AB} and the compactness in $C^2$ of bounded sets in $C^{2,\alpha}$ along with maximum principles,
	 we get convergence of $\{u_k\}$ to the desired solution $u\in C^{2,\alpha}$ on the interior and convergence of $\{\phi_k\}$ to the desired boundary function $\phi\in C^0$ on the boundary.

	\end{proof}
	
\section{Proof of Theorem \ref{2.20}}

	\begin{proof}
	
	We denote upper/lower semi-continuous functions by usc/lsc. 
	We define
	\begin{align*} A=\{u\in usc(\overline{\Omega})| F(D^2u)\geq \psi\text{ in } \Omega, \text{ u }\leq \phi \text{ on }\partial \Omega\}\\
	w(x)=\sup\{u(x)|u\in A\}.
	\end{align*}

	\begin{claim}\label{cll4.3}
	The above function $w$ is the unique continuous viscosity solution of (\ref{lab}) where $\psi$ is a constant.
	\end{claim}
	\begin{remark}  The proof follows from the following four steps. It is noteworthy that the first three steps of the proof hold good for any continuous function $\psi$. The fourth step requires a certain comparison principle (see Theorem 6.1 of Appendix), which is only available for a constant right hand side. As of now, it is unknown if such a comparison principle holds good for a continuous right hand side. In order to highlight this distinction, we present the first three steps of the proof assuming $\psi$ is any continuous function. In the final step, we assume $\psi$ to be a constant, thereby proving Theorem \ref{2.20}.
\end{remark}
	
	\begin{itemize}
	    \item[Step 1.]
	We define the following functions:
	\begin{align*}
	    \underline{z}(x)=\overline{\lim_{y\rightarrow x}}w(y)\\
	    \overline{z}(x)=\lim_{\overline{y\rightarrow x}}w(y).
	\end{align*}
	We first show that $A$ is non-empty and $w,\underline{z},\overline{z}$ are well defined. \\
	    Since $\psi\in C(\overline{\Omega})$, there exists $\varepsilon'>0$ such that $-n\frac{\pi}{2}+\varepsilon'<\psi(x)<n\frac{\pi}{2}-\varepsilon'$ for all $x\in \overline{\Omega}$. Fixing this $\varepsilon'$ we define the following functions \[
	\psi_*=-n\frac{\pi}{2}+\varepsilon'<\psi<n\frac{\pi}{2}-\varepsilon'=\psi^*.\]
	Recalling (\ref{bn}) and (\ref{bbnn}) we define 
	\begin{align}
	\underline{w}(x)=-Cx_n+ \frac{1}{2}|x|^{2}\tan\frac{\psi^*}{n}\nonumber\\
	\overline{w}(x)=Cx_n+ \frac{1}{2}|x|^{2}\tan\frac{\psi_*}{n}\label{qq}
	\end{align}
	where $C=C(||\phi||_{C^2(\partial{\Omega})}, n,|\partial\Omega|_{C^2})$. 
	 By definition $\underline{w}\in A$, which shows that $A$ is non-empty. Next, $\max\{u,\underline{w}\}$ is upper semi continuous and still a subsolution of (\ref{lab}), so we replace $u\in A$ by $\max\{u,\underline{w}\}$. This shows $u\geq \underline{w}$ and, therefore, $w$ is well defined. 
	Next, we observe that since $\underline{w},\overline{w}$ are sub and super-solutions of (\ref{lab}) respectively, we have
	\[
	\underline{w}\leq u\leq \overline{w}\] which shows $\underline{z},\overline{z}$ are well defined.

	\item[Step 2.] We show that $\underline{z}$ is a subsolution of (\ref{lab}).\\
	Suppose not. Then we can find a quadratic polynomial $P$ such that $P(x)\geq \underline{z}(x)$ in $B_{\rho}(0)$ with equality holding at $0$, such that  $F(D^2P)<\psi_*$ in $B_{\rho}(0)$. Now we choose $\varepsilon>0$ such that \begin{equation}
	    F(D^2P+4\varepsilon I)<\psi_*. \label{name11}
	\end{equation}
	From the definition of $w$ and $\underline{z}$, we can find sequences $\{u_k\}\subset A$ and $\{x_k\}\subset \Omega$, with $x_k\rightarrow 0$ such that 
	\[ \underline{z}(0)=\overline{\lim_{y\rightarrow 0}}w(y)=\lim_{x_k \rightarrow 0}u_k(x_k).
	\]
 	For $k$ large enough, we see that 
 	\begin{align*}
 	    |u_k(x_k)-P(x_k)-2\varepsilon|x_k|^2|=|u_k(x_k)-P(0)+P(0)-P(x_k)-2\varepsilon|x_k|^2|\\
 	    =o(1)<\varepsilon \rho^2.
 	\end{align*}
	On $\partial B_{\rho}(0)$, we see 
	\[ u_k(x)\leq w(x)\leq \underline{z}(x)\leq P(x)+2\varepsilon|x|^2-\varepsilon\rho^2.
	\]
	Using the definition of $w$ and $\underline{z}$, we see that for any $k$, the following holds in $B_{\rho}(0)$
	\[Q(x)=P(x)+2\varepsilon|x|^2\geq u_k(x).
	\]
	Fixing a $k$ large enough, we observe the following. The functions $u_k(x_k)$ and $Q(x_k)$ are less than $\varepsilon\rho^2$ apart, but $u_k$ is at a distance of more than $\varepsilon\rho^2$ below $Q$ on $\partial B_{\rho}(0)$. So we drop $Q$ at most $\varepsilon \rho^2$ so that it touches $u_k$ at a point inside $B_{\rho}(0)$ while still remaining above $u_k$ on $\partial B_{\rho}(0)$.
	So there exists $\gamma\leq \varepsilon\rho^2$ such that in $B_{\rho}(0)$
\[u_k(x)\leq P(x)+2\varepsilon|x|^2-\gamma
\] with equality holding at an interior point of $B_{\rho}$. Now since $u_k$ is a subsolution, we have 
\[\psi\leq F(D^2P+4\varepsilon I).
\] This contradicts (\ref{name11}).\\
Noting that $\underline{z}$ is upper semi-continuous, we see that it is a subsolution of (\ref{lab}).
\item[Step 3.] We show that $\overline{z}$ is a supersolution of (\ref{lab}).\\
Suppose not. Then we can find a quadratic polynomial $P$ such that $P(x)\leq \overline{z}(x)$ in $B_{\rho}(0)$ with equality holding at $0$, such that  $F(D^2P)>\psi^*$ in $B_{\rho}(0)$. We choose $\varepsilon>0$ small enough such that \begin{equation}
    F(D^2P-2\varepsilon I)>\psi^* .\label{name22}
\end{equation}
 We have $\overline{z}\geq P-\varepsilon|x|^2$. We define a new quadratic $Q(x)=P(x)-\varepsilon|x|^2+\varepsilon\rho^2.$ 
Observe that, since $\overline{z}(0)=\underline{\lim}_{x_k\rightarrow 0}w(x_k)$, so for $k$ large enough, we have 
\begin{align*}
    w(x_k)=\overline{z}(0)+o(1)=P(0)-P(x_k)+P(x_k)+o(1)\\
    =P(x_k)+o(1)=Q(x_k)-\varepsilon\rho^2+o(1)< Q(x_k).
\end{align*}
This contradicts the supremum definition of $w$ since $Q$ is a subsolution of (\ref{lab}) by (\ref{name22}). Noting that $\overline{z}$ is lower semi-continuous, we see that it is a supersolution of (\ref{lab}).
\item[Step 4.] We take care of the boundary value in this final step. This is where we assume (for the first time) that $\psi$ is a constant. 
Note that now we may assume the boundary value $\phi\in C^2(\partial \Omega)$ since we can always approximate $\phi$ by a sequence of smooth functions $\phi_\delta$, that solve 
\begin{align*}
    \begin{cases} F(D^2u_\delta)=\psi \text{ in $\Omega$}\\
    u_\delta=\phi_\delta \text{ on $\partial \Omega$}
    \end{cases}
\end{align*}
and apply the comparison principle\footnote{see Appendix}
to get 
\[\max_{\Omega}|u_{\delta_1}-u_{\delta_2}|\leq \max_{x \rightarrow
\partial \Omega}|(\phi_{\delta_1}-\phi_{\delta_2})(x)|\rightarrow 0
\] as $\delta_1, \delta_2 \rightarrow 0$. We have $u_\delta\rightarrow u$ in $C^0$ as $\delta\rightarrow 0$. Next, we pick an arbitrary point $x_0\in \partial \Omega$ and recall the construction of $\underline{w}, \overline{w}$ from (\ref{qq}). Defining similar functions at $x_0$ and on using the comparison principle, we get $\underline{w}\leq u\leq \overline{w}$ with equality holding at $x_0$ for all $u\in A$.
Again since $\max(u,\underline{w} )\in A$ for all $u\in A$, we can replace 
\[w(x)=\sup_{u\in A}\max (u, \underline{w}).
\] We get $\underline{w}\leq u\leq \overline{w}$ with equality holding at $x_0$,
   which shows
\[\overline{z}(x_0)= \phi(x_0)=\underline{z}(x_0).
\]
Since $x_0\in \partial \Omega$ is arbitrary, we have
$\overline{z}=\underline{z}=\phi
$ on $\partial \Omega$.
Combining the above steps and on using the comparison principle we see 
\[\overline{z}=\underline{z}=w\in C^0(\overline{\Omega})
\] is the desired solution.  
This proves the existence part of claim (\ref{cll4.3}). Uniqueness again
follows from the comparison principle
	\end{itemize}

	\end{proof}

	\section{Appendix}
	We state the following linear algebra Lemma that was used in proving the double normal estimate in step 3.4 of section 3. 
	\begin{lemma}\cite[Lemma 1.2]{CNS}\label{3.9}
			Consider the following $n \times n$ symmetric matrix
			$$ M=\begin{bmatrix}
\lambda'_1 &  &  &  & a_1\\ 
           & . & &  & .\\
           &  &. &  & .\\
            &  &  & \lambda'_{n-1} & a_{n-1}\\
           a_1 & . & . & a_{n-1} & a\\  
\end{bmatrix}.
$$\\
		where $\lambda'_1, \lambda'_2,..,\lambda'_{n-1}$ are fixed, $|a_i|<C$ for $1\leq i\leq n-1$, and $|a|\rightarrow+\infty$. Then the eigenvalues $\lambda_1,\lambda_2,...,\lambda_n$ of $M$ behave like 
			\begin{equation*}
			\lambda'_1+o(1),\lambda'_2+o(1),..., \lambda'_n+o(1),a+O(1)
			\end{equation*}
			where $o(1)$ and $O(1)$ are uniform as $a\rightarrow\infty$.
		\end{lemma}

	For the sake of completeness we state and prove the well known comparison principle for strictly elliptic equations\footnote{We learned this proof from \cite{yy2004}. }. 
\begin{theorem}\label{A}
Suppose that $u$ is a usc subsolution and $v$ is a lsc supersolution of the strictly elliptic equation (\ref{s1}) in $\Omega\subset\mathbb{R}^n$. If $u\leq v$ on $\partial \Omega $, then $u\leq v$ in $\Omega$.
        \end{theorem}		
        
        \begin{proof}
    W.l.o.g we assume $\Omega=B_1(0)$ and $u\leq v-2\delta$ on $\partial B_1$ for some small $\delta>0$. We re-write equation (\ref{s1}) as \[F(D^2u)=\sum _{i=1}^{n}\arctan \lambda_{i}-c=0.
    \]
    Let $u^{\varepsilon}$ be an upper parabolic envelope\footnote{For $\varepsilon>0$, we define the upper $\varepsilon$-envelope of $u$ to be
        \[u^{\varepsilon}(x_0)=\sup_{x\in \overline{H}}\{u(x)+\varepsilon-\frac{1}{\varepsilon}|x-x_0|^2\}, \text{    for }x_0\in H
        \] where $H$ is an open set such that $\overline{H}\subset B_1$.} satisfying
        \begin{align*}
            F(D^2u^{\varepsilon})\geq 0 \\
            D^2u^{\varepsilon}\geq -C/\varepsilon\\
            ||u^{\varepsilon}||_{C^{0,1}}\leq C/\varepsilon
            \end{align*} outside a measure zero subset where $u^{\varepsilon}$ is punctually second order differentiable and $C$ is chosen such that $u^{\varepsilon}-v_{\varepsilon}\leq C-\varepsilon|x-x_0|^2$ on $\partial B_1$ with equality holding at $x_0\in B_1.$ We see that $0\leq u^{\varepsilon}(x)-u(x)\leq u(x^*)-u(x)+\varepsilon$ where $x^*\rightarrow x$ as $\varepsilon\rightarrow 0$.
		By symmetry, the lower parabolic envelope $v_{\varepsilon}$ satisfies
		\begin{align*}
            F(D^2v_{\varepsilon})\leq 0 \\
            D^2v_{\varepsilon}\leq C/\varepsilon\\
            ||v_{\varepsilon}||_{C^{0,1}}\leq C/\varepsilon
            \end{align*} and $0\geq v_{\varepsilon}(x)-v(x)\geq v(x_*)-v(x)-\varepsilon$ where $x_*\rightarrow x$ as $\varepsilon\rightarrow 0$. Note that $v_{\varepsilon}-u^{\varepsilon}\leq L+\frac{C}{\varepsilon}|x-x_0|^2$ for $x_0\in B_1$  where $L$ is a linear function. The convex envelope $\Gamma(v_{\varepsilon}-u^{\varepsilon})$ is in $C^{1,1}$. From Alexandroff's estimate we have 
            \[\sup_{B_1}(v_{\varepsilon}-u^{\varepsilon})^{-}\leq C(n)[\int_{\Sigma}\det D^2\Gamma]^{1/n}
            \] where $\Sigma=\{x\in B_1|\Gamma(x)=v_{\varepsilon}(x)-u^{\varepsilon}(x)\}$. Now in $\Sigma$, we have $0\leq D^2\Gamma\leq D^2(v_{\varepsilon}-u^{\varepsilon})$ or $L(x)\leq v_{\varepsilon}(x)-u^{\varepsilon}(x)$ near $x_0\in \Sigma$. For $K$ large since $u^{\varepsilon}+\frac{K}{\varepsilon}|x|^2$ is convex and $v_{\varepsilon}-\frac{K}{\varepsilon}|x|^2$ is concave, we have the following for a.e. $x_0\in B_1$
            \begin{align*}
                v_{\varepsilon}=\Gamma+\frac{K}{\varepsilon}|x|^2+O(|x-x_0|^2)\\
                u^{\varepsilon}=\Gamma+\frac{K}{\varepsilon}|x|^2+O(|x-x_0|^2).
            \end{align*}
            Again since $v_{\varepsilon}$ is a super solution and $u^{\varepsilon}$ is a sub solution, for a.e. $x_0\in B_1$, we have \begin{align*}
                F(D^2v_{\varepsilon}(x_0))\leq 0\\
                F(D^2u^{\varepsilon}(x_0))\geq 0\\
                F(D^2v_{\varepsilon}(x_0))-F(D^2u^{\varepsilon}(x_0))\leq 0.
            \end{align*}
            Also, a.e. $x_0\in \Gamma$, we have $D^2v_{\varepsilon}(x_0)-D^2u^{\varepsilon}(x_0)\geq 0$. 
            However, $F$ is strictly elliptic, so we must have $F(D^2v_{\varepsilon})-F(D^2u^{\varepsilon})\geq 0$, which shows
            \begin{align*}
                F(D^2v_{\varepsilon}(x_0))=F(D^2u^{\varepsilon}(x_0))\text{ a.e } x_0\in \Sigma.
            \end{align*} Again, given that $F$ is strictly elliptic, the line with the positive direction $D^2v_{\varepsilon}(x_0)-D^2u^{\varepsilon}(x_0)$ intersects the level set $\{F=C\}$ only once, which implies $D^2v_{\varepsilon}(x_0)=D^2u^{\varepsilon}(x_0)$. This shows $
            \sup_{B_1}(v_{\varepsilon}-u^{\varepsilon})^{-}\leq 0$, which proves that
            \[v\geq v_{\varepsilon}\geq u^{\varepsilon}\geq u \text{ in } B_1.
            \]
            \end{proof}

	\bibliographystyle{amsalpha}
	\bibliography{name}
	
\end{document}